\documentclass[a4paper,12pt]{amsart}

\usepackage{amssymb}
\usepackage{hyperref}
\usepackage[normalem]{ulem}

\def\altdb{\vadjust{\vbox to 0pt{\vss\hbox{\kern \hsize
\quad{\dbend}}\kern\baselineskip\kern-10pt}}}

\newcommand\field[1]{\mathbb{#1}}
\newcommand\CC{\field{C}}

\newcommand\Gg{G}

\renewcommand\ker{\operatorname{ker}}

\newcommand\supp{\operatorname{supp}}

\theoremstyle{plain}
\newtheorem{theorem}{Theorem}[section]
\newtheorem*{theorem*}{Theorem}
\newtheorem*{prop*}{Proposition}
\newtheorem{cor}[theorem]{Corollary}
\newtheorem{lemma}[theorem]{Lemma}
\newtheorem{prop}[theorem]{Proposition}

\theoremstyle{remark}
\newtheorem{rmk}[theorem]{Remark}

\theoremstyle{definition}

\newcommand{\go}{\Gg^{(0)}}
\newcommand{\Bco}[2]{B^{\operatorname{co}}_{#1}(#2)}
\newcommand{\BcoG}{\Bco{*}{\Gg}}
\newcommand{\BncoG}{\Bco{n}{\Gg}}

\numberwithin{equation}{section}

\title{Uniqueness Theorems for Steinberg Algebras}

\author{Lisa Orloff Clark}
\author{Cain Edie-Michell}

\address{Lisa Orloff Clark and Cain Edie-Michell\\ Department of Mathematics and Statistics\\ University of Otago\\ PO Box 56\\ Dunedin 9054\\ NEW ZEALAND}
\email{lclark@maths.otago.ac.nz, edica501@student.otago.ac.nz}



\begin{document}

\begin{abstract}
We prove Cuntz-Krieger and graded uniqueness theorems for Steinberg algebras.  We also show that a Steinberg algebra is
basically simple if and only if its associated groupoid is both effective and minimal.  Finally
we use results of Steinberg to characterise the center
of Steinberg algebras associated to minimal groupoids.
\end{abstract}

\maketitle

\section{Introduction}\label{sec:int}
Let $G$ be a Hausdorff, ample groupoid and $R$ a commutative ring with identity. The Steinberg algebra $A_R(G)$
is the $R$-algebra of locally constant functions  from $G$ to $R$ with compact support. These algebras were introduced in
 \cite{Steinberg2010} as a model for discrete inverse semigroup algebras.
Steinberg algebras also include the Kumjian-Pask algebras of \cite{ACaHR} and \cite{CFaH} 
 which themselves include the Leavitt path algebras of \cite{Tomforde}.

In this paper we generalise a number of results about complex Steinberg algebras from \cite{BCFS} and \cite{CFST} to 
Steinberg algebras
over arbitrary commutative rings with identity.  We also investigate the center of $A_R(G)$.

In section~\ref{sec:uniqueness},  we prove Cuntz-Krieger and graded uniqueness theorems for $A_R(G)$.  
These theorems generalise the analogous graph algebra uniqueness theorems and give conditions
under which an $R$-algebra homomorphism $\pi:A_R(G) \to A$ is injective.  
Complex Steinberg algebra versions of these
theorems are given in \cite{CFST} but the proofs rely on the standard inner product in $\mathbb{C}$.
We noticed
that the analytic structure of $\mathbb{C}$ is not required and we simplify the arguments considerably.   

In section~\ref{sec:simp}, we generalize \cite[Theorem~4.1]{BCFS}, which says a complex Steinberg 
algebra is simple if and only if $G$ is \emph{effective} and \emph{minimal}.
This result is not true for more general Steinberg $R$-algebras  because ideals in $R$ create ideals in $A_R(G)$.
However, as was done first for Leavitt path algebras in \cite{Tomforde} and then for Kumjian-Pask algebras in \cite{ACaHR}, 
we define \emph{basic simplicity} for Steinberg algebras, and show that 
$A_R(G)$ is basically simple if and only if $G$ is effective and minimal.

In the final section, we investigate the center of $A_R(G)$. Steinberg gives a complete characterisation of 
the center of $A_R(G)$ in \cite[Proposition~4.13]{Steinberg2010}.   We apply his result to algebras associated to
 minimal groupoids in Theorem~\ref{thm:center}.  This theorem  is a Steinberg algebra generalisation of the 
Kumjian-Pask algebra result of \cite[Theorem~4.7]{center}.  

\smallskip

\textbf{Acknowledgement.}  Thank you to Aidan Sims for a number of helpful conversations.

\smallskip

\section{Preliminaries}\label{sec:prelim}

\subsection{Groupoids}A \emph{groupoid} consists of sets $G$ and  $G^{(2)}\subseteq G \times G$ 
such that there is a map $(\alpha,\beta) \mapsto \alpha \beta$ from $G^{(2)} \to G$ (sometimes 
called composition) and an
involution $\alpha \mapsto \alpha^{-1}$ on $G$, such that the following conditions hold:
\begin{enumerate}
\item if $(\alpha,\beta)$ and $(\beta,\gamma)$ are in $G^{(2)}$, then so are $(\alpha \beta,\gamma)$ and $(\alpha,\beta \gamma)$, and the equation,
$(\alpha \beta)\gamma = \alpha (\beta \gamma)$ holds;
\item for all $\alpha \in G$, $(\alpha^{-1},\alpha) \in G^{(2)}$ and if $(\alpha,\beta)\in G^{(2)}$, then $\alpha^{-1}(\alpha \beta) = \beta$ and
$(\alpha \beta)\beta^{-1} = \alpha$.
\end{enumerate}
We define the functions $r$ and $s$ from $G$ to itself by \[r(\alpha) = \alpha\alpha^{-1} \text{ and } s(\alpha) = \alpha^{-1}\alpha.\] 
These are called the range and source maps respectively. We call the common image of $r$ and $s$ the unit space of $G$ and denote it $G^{(0)}$.
Note that $(\alpha, \gamma) \in G^{(2)}$ if and only if $s(\alpha) = r(\gamma)$.
We define the product of two subsets $A$ and $B$ in $G$ by
\begin{equation*}
AB := \{\alpha\beta : \alpha \in A, \beta \in B, s(\alpha) = r(\beta)\}.
\end{equation*}
We will let $\operatorname{Iso}(G)$ denote the \emph{isotropy subgroupoid} of $G$. That is \[
\operatorname{Iso}(G) := \{\gamma \in G \mid s(\gamma)=r(\gamma)\}.\]

For any $u \in \go$, we define the \emph{orbit} of $u$ 
\[
 [u] := s(r^{-1}(u)) = r(s^{-1}(u)) \subseteq \go.
\]
We say a subset $U$ of $\go$ is \emph{invariant} if $r(s^{-1}(U)) = U$.

We call $G$ a \emph{topological groupoid} if $G$ is endowed with a topology such that composition and inversion are continuous. 
An \emph{ open bisection} of $G$ is an open subset $B$ of $G$ such that both $r|_B$ and $s|_B$ are homeomorphisms.  
We call $G$ \emph{\'etale} if  $s$
 is a local homeomorphisms; in this case, $r$ is also a local homeomorphism.  
Notice that $G$ is \'etale if and only if $G$ has a basis of open bisections.   
We call $G$ \emph{ample} if $G$ has a basis of compact open bisections.  We also have that a locally compact, Hausdorff \'etale
groupoid is ample if and only if $\go$ is totally disconnected (see \cite[Proposition~4.1]{Exel:PAMS2010}).
In general, we
will only be considering Hausdorff, ample groupoids in this paper.

We say $G$ is \emph{effective} if the interior of Iso($G$) $\setminus G^{(0)}$ is empty.  
That is, $G$ is effective if every open subset $U \subseteq G \setminus \go$  has an element $\gamma \in U$ such that
   $s(\gamma) \neq r(\gamma)$.  Since we do not assume that
$G$ is second countable, effective is strictly weaker than  `topologically principal' (see 
\cite[Example~6.4]{BCFS}).  We call $G$ \emph{minimal} if $G^{(0)}$ has no nontrivial open invariant subsets.  
Equivalently, $G$ is minimal if each orbit $[u]$ is dense in $\go$.

\subsection{Steinberg Algebras}
  Suppose $G$ is a Hausdorff, ample groupoid and  $R$ is a commutative ring with identity.  Let $A_R(G)$ be the set of all functions
from $G$ to $R$ that are locally constant and have compact support.  It is easy to check that if $f \in A_R(G)$, then 
$\supp f$ is clopen.

\begin{rmk}
Note that if $f : G \to R$ is locally constant, then it is continuous with respect to any topology on $R$. To see this, observe
that for $r \in R$ and $\alpha \in f^{-1}(r)$ there is a neighbourhood $U$ of $\alpha$ such that $f|_U \equiv r$, therefore
$f^{-1}(r)$ is open. Thus for any set $S \subseteq R$, $f^{-1}(S) = \bigcup\limits_{r\in S} f^{-1}(r)$ is open.
\end{rmk}
We define addition in $A_R(G)$ pointwise, and for $f,g \in A_R(G)$ convolution is defined by
\begin{equation*}
(f*g)(\alpha) = \sum\limits_{r(\beta) = r(\alpha)} f(\beta)g(\beta^{-1}\alpha).
\end{equation*}
With this structure, $A_R(G)$ is an algebra called the \emph{Steinberg algebra} associated to $G$.
The algebra $A_R(G)$ can also be realised as the span of characteristic functions of the form $1_B$ where
$B$ is a compact open bisection (see \cite[Lemma~3.3]{CFST}).   Convolution of characteristic functions is 
nicely behaved:
\[
 1_B * 1_D = 1_{BD}
\]
for compact open bisections $B$ and $D$ (see \cite[Proposition~4.5]{Steinberg2010}). Note 
 that $A_R(G)$ is unital if and only if $\go$ is compact (see \cite[Proposition 4.11]{Steinberg2010}).
Steinberg algebras were introduced in \cite{Steinberg2010} in the context of discrete inverse semigroup
algebras.\footnote{Steinberg uses the notation $`RG'$ rather than $A_R(G)$.}  
In \cite[Example~3.2]{morita}, it is shown that Steinberg algebras include Leavitt path
algebras.  More generally, it is shown in \cite[Proposition~4.3]{CFST} that Steinberg algebras include Kumjian-Pask algebras.

Let $\Gamma$ be a discrete group (with identity $e$) and $c : G \to \Gamma$ a continuous cocycle;
that is, $c$ preserves composition in $G$.  For each $n \in \Gamma$ we write
 $G_n := c^{-1}(n)$ and write $A_R(G)_n$ for the subset of $A_R(G)$ consisting of functions whose support is contained in 
$G_n$.  With this
structure, $A_R(G)$ is graded by $\Gamma$ (see \cite[Proposition~3.6]{CFST}).  
Note that any Hausdorff, ample groupoid admits the trivial cocycle from
$G$ to the trivial group $\{e\}$ which yields the trivial grading on $A_R(G)$.

We say a subset $S$ of $G$ is $n$-graded if $S \subseteq G_n$ and we write $\BncoG$ for the collection of all 
$n$-graded compact open bisections of $G$. We write \[\BcoG:=
\bigcup\limits_{n\in\Gamma} \BncoG.\]  The following lemma  is an R-algebra generalisation of \cite[Lemma~3.5]{CFST} and the 
proof is the same.

\begin{lemma}
 \label{lem:3.5}
Let $G$ be a Hausdorff, ample groupoid, $\Gamma$ a discrete group, and $c : G \to \Gamma$ a continuous cocycle.
Every $f \in A_R(G)$ can be expressed as \[f = \sum_{B \in F} a_B 1_{B}\] where 
$F$ is a finite set of mutually disjoint elements of $\BcoG$ and each $a_B \in R$.
\end{lemma}

Next we describe how $A_R(G)$ is a universal algebra.  Let $A$ be an $R$-algebra. A \emph{representation of $\BcoG$} in $A$ is a family \[\{t_B : B
\in \BcoG\} \subseteq A\] satisfying
\begin{enumerate}\renewcommand{\theenumi}{R\arabic{enumi}}
\item\label{it:zero} $t_\emptyset = 0$;
\item\label{it:multiplicative} $t_Bt_D = t_{BD}$ for all $B,D \in \BcoG$; and
\item\label{it:additive} $t_B + t_D = t_{B \cup D}$ whenever $B$ and $D$ are disjoint
    elements of $\BncoG$ for some $n$ such that $B \cup D$ is a bisection.
\end{enumerate}

\begin{prop}\label{thm:uni}
Let $G$ be a Hausdorff, ample groupoid, $\Gamma$ a discrete group, and $c : G \to \Gamma$ a continuous cocycle.
Then $\{1_B : B \in \BcoG\} \subseteq A_R(G)$ is a representation of $\BcoG$
which spans $A_R(G)$. Moreover, $A_R(G)$ is universal for representations of $\BcoG$ in the sense that
for every representation $\{t_B : B \in \BcoG\}$ of $\BcoG$ in an $R$-algebra $A$, there is a unique $R$-algebra
homomorphism $\pi : A_R(G) \to A$ such that $\pi(1_B) = t_B$ for all $B \in \BcoG$.
\end{prop}

Once again, the proof of  is exactly the same as the proof of \cite[Theorem 3.10]{CFST} with
$A(G)$ replaced with $A_R(G)$.

\subsection{Basic ideals}  We call an ideal $I$ in $A_R(G)$ a \emph{basic ideal}  
if it satisfies the following:
if $r1_K \in I$ for  $r \in R$ and compact open
$K \subseteq G^{(0)}$, then $1_K \in I$. 
We say $A_R(G)$ is \emph{basically simple} if it contains no non-trivial
basic ideals. This definition is a generalization of the notion of a basic simplicity for graph algebras
from in \cite{Tomforde} and \cite{ACaHR}.

\section{The uniqueness theorems}\label{sec:uniqueness} In this section, we prove Cuntz-Krieger and
graded uniqueness theorems for Steinberg algebras.
Our proofs are similar in structure to the analogous proofs in section 5 of \cite{CFST} where $R = \CC$.
However in \cite{CFST} the authors use the inner-product of $\mathbb{C}$ to build a function
in the Steinberg algebra that is nonzero on $\go$.
In the following key lemma, we produce the desired function in our more general setting.

\begin{lemma}\label{thm:grad}
Let $G$ be a Hausdorff, ample groupoid, $\Gamma$ a discrete group, and $c : G \to \Gamma$  a continuous cocycle. Fix
 $k \in \Gamma$.  Then for each $g_k \in A_R(G)_k \setminus \{0\}$, there exists a compact open bisection $B$ such that
 $f := 1_{B^{-1}} * g_k \in A_R(G)_e$ and $\supp(f) \cap G^{(0)} \neq \emptyset$.
\end{lemma}

\begin{proof}
Let $g_k \in A_R(G)_k$.   Using Lemma~\ref{lem:3.5}, we can write \[g_{k}=\sum_{D \in F} a_D 1_{D}\] where
$F$ is a finite set of mutually disjoint elements of $\Bco{k}{\Gg}$ and each $a_D \in R$.
As $g_k \neq 0$ there must exist $B \in F$ such that $a_B \neq 0$. Now define $f := 1_{B^{-1}} * g_k$. We claim that
$ f \in A_R(G)_e$. To see this notice that

\begin{equation*}
f = 1_{B^{-1}} * g_k = 1_{B^{-1}} * \sum\limits_{D \in F} a_D1_D =  
\sum\limits_{D \in F}a_D1_{B^{-1}} * 1_D = \sum\limits_{D \in F}a_D1_{B^{-1}D}.
\end{equation*}
Because each $D \in F$ is a subset of $G_k$, we have
\begin{equation*}
B^{-1}D \subseteq G_{k^{-1}k} = G_e,
\end{equation*}
thus $f\in A_R(G)_e$ as claimed.

To show $f$ is non-zero on $G^{(0)}$ let $\alpha \in B$. We claim that $\alpha^{-1}\alpha \not\in B^{-1}D$ for any $D \in F \setminus B$.
 To prove this, suppose by way of contradiction that for some $D\in F \setminus B$, there exists $x \in B$ and $y \in D$ such that
 $x^{-1}y = \alpha^{-1}\alpha$, but then $y = x\alpha^{-1}\alpha = x$, which is a contradiction as $B$ and $D$ are mutually disjoint.
Thus $1_{B^{-1}D}(\alpha^{-1}\alpha) = 0$ for all $D\in F \setminus B$, and therefore
\begin{equation*}
f(\alpha^{-1}\alpha) = \sum\limits_{D \in F} a_D1_{B^{-1}D}(\alpha^{-1}\alpha) = a_B \neq 0.
\end{equation*}
\end{proof}

We are now ready to prove the Cuntz-Krieger uniqueness theorem.  
Notice that we generalise the complex version of \cite[Theorem~5.1]{CFST} in two ways.
First, we are working over an arbitrary commutative ring with identity.  
Second, we have replaced the hypothesis that $G$ be topologically
principal with the weaker requirement that $G$ be effective.  The observation that effectiveness suffices comes
from \cite[Lemma 4.2]{BCFS}.

\begin{theorem}\label{thm:CKUT}
Let $G$ be an effective Hausdorff, ample groupoid, and $R$ a commutative ring with identity.  
Let $\pi : A_R(G) \to A$ be an $R$-algebra homomorphism. Suppose that $\ker(\pi) \neq \{0\}$. 
Then there is a compact open subset $K \subseteq G^{(0)}$ and $r \in R \setminus\{0\}$ such that $\pi(r1_K) = 0$.
\end{theorem}

\begin{proof}
Fix $g \in \ker(\pi) \setminus \{0\}$. We can use Lemma~\ref{thm:grad} with the  trivial cocycle  to
 get a compact open bisection $B$ such that $f := 1_{B^{-1}} * g$ is nonzero on $G^{(0)}$. Notice  that $f \in \ker(\pi)$ 
because $\pi$ is a homomorphism.
Define the nonzero function $f_0$ such that \[
   f_0(\alpha) = \begin{cases}
       f(\alpha), &  \text{ if } \alpha \in G^{(0)};\\
       0, &  \text{otherwise}.\end{cases}
\]
 Because $G^{(0)}$ is both open and closed, $f_0 \in A_R(G)$.
Using Lemma~\ref{lem:3.5}, write\begin{equation*}
f_0 = \sum\limits_{D \in F} a_D1_D
\end{equation*}
where $F$ is a collection of mutually disjoint, nonempty compact open subsets of $G^{(0)}$.
Let $H$ be the support of $f - f_0$. Notice that $H \subseteq G \setminus G^{(0)}$.  Fix $D_0 \in F$.
As $G$ is effective we can apply \cite[Lemma 3.1 (1 $\to$ 4)]{BCFS} to get nonempty open
$K \subseteq D_0$ such that $KHK = \emptyset$. Since $G$ has a basis of compact open sets,
we can also assume $K$ is compact.

Now compute for $\gamma \in G$ 
\begin{equation*}
(1_K*(f-f_0)*1_K)(\gamma) = 1_K(r(\gamma))(f-f_0)(\gamma)1_K(s(\gamma))  = 0.
\end{equation*}
Thus the linearity of convolution gives \[1_K*f*1_K = 1_K*f_0*1_K\]
 which is equal to $a_{D_0}1_K$.
Thus $r :=a_{D_0}$ satisfies
\begin{equation*}
\pi(r1_K) = \pi(1_K)\pi(f)\pi(1_K) = 0.
\end{equation*}

\end{proof}

From this theorem we get the following corollary which will be useful in the next section.

\begin{cor}\label{thm:CKUT4}
Let $G$ be an effective, Hausdorff, ample groupoid and $R$ a commutative ring with identity.
Suppose that $I$ is a nontrivial basic ideal of $A_R(G)$. Then there is a
compact open subset $K \subseteq G^{(0)}$ such that $1_K \in I$.
\end{cor}
\begin{proof}
Let $\pi : A_R(G) \to A_R(G) / I$ be the quotient map. We can apply Theorem~\ref{thm:CKUT} to get
compact open $K \subseteq G^{(0)}$ and $r \in R \setminus\{0\}$ such that $\pi(r1_K) = 0$.
As $\ker(\pi) = I$ we have that $r1_K \in I$.  Since $I$ is a basic ideal we get $1_K \in I$.
\end{proof}

The graded uniqueness theorem follows from a bootstrapping argument.

\begin{theorem}\label{thm:CKUT2}
Let $G$ be a Hausdorff, ample groupoid, $R$ a commutative ring with identity, $\Gamma$ a discrete group, 
and $c : G \to \Gamma$ a continuous cocycle. Suppose that $G_e$ is effective. Let $\pi : A_R(G) \to A$ be a 
graded $R$-algebra homomorphism. Suppose that $\ker(\pi) \neq \{0\}$. Then there is a compact open subset 
$K \subseteq G^{(0)}$ and $r \in R \setminus \{0\}$ such that $\pi(r1_K) = 0$.
\end{theorem}

\begin{proof}
First we claim that there exists a nonzero $f \in A_R(G)_e$ such that $\pi(f) = 0$. To see this, observe
that since $\ker (\pi) \neq \{0\}$, there exists $g \in \ker(\pi)\setminus\{0\}$ such that $\pi(g) =
0$. Since $g$ is an element of the graded algebra $A_R(G)$, $g$ can be expressed as a finite sum of
graded components $g = \sum_{h \in F}  g_h$ where $F \subseteq \Gamma$ and each $g_h \in A_R(G)_h$. Now
\[\pi(g) = \sum_{h \in F} \pi(g_h) = 0,\] and each $\pi(g_h) \in A_h$ because $\pi$ is a graded
homomorphism. Because the graded subspaces of $A$ are linearly independent, it follows that each
$\pi(g_h) = 0$. Since $g \not= 0$, there exists $k \in F$ such that $g_{k} \not= 0$.  Now we apply 
Lemma~\ref{thm:grad} to get a compact open bisection $B$ such that 
$f := 1_{B^{-1}} * g_k \in A_R(G)_e \setminus \{0\}$. We have $\pi(f) = 0$ as $\pi(g_k) = 0$.

By hypothesis  $G_e$ is effective. By definition, $A_R(G)_e$ is equal
to the space of locally constant, continuous, compactly supported functions on $\Gg_e$, so we may
apply Theorem~\ref{thm:CKUT} to the restricted homomorphism to get $K \subseteq G^{(0)}$ and $r \in R \setminus \{0\}$ such that
 $\pi$ annihilates $r1_K$.
\end{proof}
\section{Basic Simplicity}\label{sec:simp}

In this section, we prove the following theorem.

\begin{theorem}\label{thm:CKUT8}
Let $G$ be an Hausdorff, ample groupoid, and $R$ a commutative ring with identity. Then $A_R(G)$ is 
basically simple if and only if $G$ is effective and minimal.
\end{theorem}

Our proof of Theorem~\ref{thm:CKUT8} is similar in structure to the proof of
complex version of \cite[Theorem 4.1]{BCFS}. 
However, before we proceed with the proof, we need to generalise the $\CC$-algebra representation built in \cite[Proposition 4.3]{BCFS}. 
Given an invariant subset $W$ of $G^{(0)}$ and a commutative ring $R$ with identity, we write $\mathbb{F}_R(W)$ for the free $R$-module with basis $W$. 
\begin{prop}\label{thm:homo}
Let $G$ be a Hausdorff, ample groupoid, $R$ a commutative ring with identity, and $W$ an invariant subset of 
$G^{(0)}$.
\begin{enumerate}
\item \label{it1:rep} For every compact open bisection $B \subseteq G$ there is a function $f_B : G^{(0)} \to \mathbb{F}_R(W)$ that 
has support contained in $s(B)\cap W$ and satisfies $f_B(s(\gamma)) = r(\gamma)$ for all $\gamma \in B \cap s^{-1}(W)$.
\item \label{it2:rep} There is a unique homomorphism $\pi_W : A_R(G) \to End(\mathbb{F}_R(W))$ such that \[\pi_W(1_B)(u) = f_B(u)\]
 for every compact open bisection $B$ and all $u\in W$.
\end{enumerate}
\end{prop}

\begin{rmk} 
We point out that there is a typo in \cite[Theorem~4.3]{BCFS}.  The intersection with $W$ and with $s^{-1}(W)$ were
mistakenly omitted in item (\ref{it1:rep}).
\end{rmk}

\begin{proof}
The proof of this proposition is the same as the proof of \cite[Proposition 4.3]{BCFS} with 
$\mathbb{F}(W)$ replaced with $\mathbb{F}_R(W)$, and the universal property of $A(G)$ (\cite[Theorem 3.10]{CFST}) 
replaced with Proposition~\ref{thm:uni}.
\end{proof}

\begin{rmk}
As stated in \cite{BCFS}, the map $\pi_{G^{(0)}}$ is a generalization of the infinite-path representation of a
 Kumjian-Pask algebra as defined in \cite{ACaHR}.
\end{rmk}

\begin{lemma}\label{thm:CKUT6}
Let $G$ be a Hausdorff, ample groupoid, $R$ a commutative ring with identity, and $W$ an invariant subset of 
$G^{(0)}$. Let $\pi_W$ be the homomorphism of Lemma~\ref{thm:homo}. 
Then $\ker(\pi_W)$ is a basic ideal.
\end{lemma}

\begin{proof}
Let  $K\subseteq G^{(0)}$ be a compact open set. 
Suppose that $1_K \notin \ker(\pi_W)$ and fix $r \in R\setminus \{0\}$.  It suffices to show $r1_K \notin \ker(\pi_W)$. 
Notice from Proposition~\ref{thm:homo} that
\begin{equation*}
\pi_W(1_K)(u) = 
     \begin{cases}
       u, & \text{if } u \in K \cap W; \\
       0, &  \text{otherwise}
     \end{cases}   
\end{equation*}
for each $u \in \go$.  
Since $1_K \notin \ker(\pi_W)$, there exists $u \in K \cap W$ such that \[\pi_W(1_K)(u) = u \neq 0\]
and hence 
\[\pi_W(r1_K)(u) = ru \neq 0.\]
Therefore $r1_K \notin \ker(\pi_W)$ and $\ker(\pi_W)$ is a basic ideal.
\end{proof}

\begin{proof}[proof of Theorem~\ref{thm:CKUT8}]

Suppose $G$ is not effective. Then there exists a nonempty compact
 open bisection $B \subseteq G \backslash G^{(0)}$ such that for every $\gamma \in B$, $r(\gamma) = s(\gamma)$. 
Hence $B \neq s(B)$ but $f_B = f_{s(B)}$, with $f_B$ as in Lemma~\ref{thm:homo}. Now let $\pi_{G^{(0)}}$ be the
 nonzero homomorphism of Lemma~\ref{thm:homo}.  Then 
\[\pi_{G^{(0)}}(1_B) = \pi_{G^{(0)}}(1_{s(B)})\] and hence $1_B - 1_{s(B)} \in \ker(\pi_{G^{(0)}})$. But as 
$B \neq s(B)$ we have $1_B - 1_{s(B)} \neq 0$, thus $\ker(\pi_{G^{(0)}})$ is a non-trivial basic ideal of 
$A_R(G)$ from the above and Lemma~\ref{thm:CKUT6}. Therefore $A_R(G)$ is not basically simple.

Next suppose $G$ is not minimal. Let $U$ be a nontrivial open invariant subset of $G^{(0)}$. The complement
 $W := G^{(0)} \backslash U$ is also an invariant subset of $G^{(0)}$. Let $\pi_W$ be the nonzero homomorphism of 
Lemma~\ref{thm:homo}. We have that $\ker(\pi_W)$ is a basic ideal from 
Lemma~\ref{thm:CKUT6} and since $\pi_W$ is nonzero, $\ker(\pi_W)$ is not all of $A_R(G)$. 
Let $V \subseteq U$ be a compact open nonempty set, then 
$1_V \in \ker(\pi_W)$ as $W \cap V = \emptyset$. Hence $\ker(\pi_W)$ is a nontrivial basic ideal of $A_R(G)$,
 so $A_R(G)$ is not basically simple.

Now suppose $G$ is minimal and effective and let $I \neq \{0\}$ be a basic ideal of $A_R(G)$.  Apply
Corollary~\ref{thm:CKUT4} to get a compact open subset $U \subseteq G^{(0)}$ such that $1_U \in I$.
Let $g\in A_R(G)$, and define $K := r(\supp(g)) \subseteq G^{(0)}$. As $g$ is locally constant we have that 
$K$ is compact and open. Since $s(r^{-1}(U))$ is a nonempty open invariant set, it is all of $G^{(0)}$ as $G$ is 
minimal. Therefore $K \subseteq s(r^{-1}(U))$. So for each $u \in K$, there exists $\gamma _u$ with $r(\gamma _u) \in U$ 
and $s(\gamma _u) = u$. For each $u$, let $B_u$ be a compact open bisection containing $\gamma _u$ such that
$r(B_u) \subseteq U$ and $s(B_u) \subseteq K$. Then \[1_{s(B_u)} = 1_{B_u^{-1}} * 1_U * 1_{B_u}\] belongs to $I$. 
Since $K$ is compact, there is a finite subset $\{v_1, \ldots, v_n\}$ of $K$ such that $\{s(B_{v_i}) : 1 \leq i \leq n\}$
 covers $K$. Using \cite[Remark 2.4]{CFST} we can assume that the $s(B_{v_i})$ are mutually disjoint. For each $i$, the function $k_
i := 1_{s(B_{v_i})} \in I$, so $1_K = \sum\limits_{i = 1}^{n} k_i \in I$. Hence $g = 1_K * g \in I$, so 
$I = A_R(G)$, which implies $A_R(G)$ is basically simple.
\end{proof}

\begin{cor}\label{thm:feild}
Let $G$ be a Hausdorff, ample groupoid and $R$ a commutative ring with identity. 
Then $A_R(G)$ is simple if and only if $G$ is effective and minimal, and $R$ is a field.
\end{cor}

\begin{proof}
Suppose $G$ is effective and minimal and $R$ is a field.  
Apply Theorem~\ref{thm:CKUT8} to see that $A_R(G)$ is basically simple. 
Let $I$ be a nonzero ideal of $A_R(G)$.  We claim that $I$ is a basic ideal. 
To prove the claim, suppose that $r1_K \in I$ for some 
$r \in R$ and $K\subseteq G^{(0)}$ a compact open set.  As $R$ is a 
field, there exists $r^{-1} \in R$ such that $1_K = r^{-1}(r1_K) \in I$. Thus every ideal in 
$A_R(G)$ is a basic ideal, so $A_R(G)$ is simple.

Now suppose that $A_R(G)$ is simple, then it is also basically simple so we apply 
Theorem~\ref{thm:CKUT8} to get that $G$ is effective and minimal. By way of contradiction, 
suppose $I$ is a nontrivial ideal of $R$. Then $IA_R(G)$ is a nontrivial ideal of $A_R(G)$ which
is a contradiction.  Thus $R$ is a field.
\end{proof}
\section{The center of $A_R(G)$}\label{sec:cent}

Following \cite[Definition~4.12]{Steinberg2010}, we say $f \in A_R(G)$ is a class function if:
\begin{enumerate}
\item \label{it1:classfunction} $f(\alpha) \neq 0$ implies $r(\alpha) = s(\alpha)$; and
\item \label{it2:classfunction} $s(\alpha) = r(\alpha) = s(\beta)$ implies $f(\beta\alpha\beta^{-1}) = f(\alpha)$.
\end{enumerate}
Steinberg shows that the center $Z(A_R(G))$ of $A_R(G)$ is the set of \emph{class functions} 
(see \cite[Proposition~4.13]{Steinberg2010}).

\begin{lemma}\label{thm:inv}
Let $G$ be a Hausdorff, ample groupoid, and $R$ a commutative ring with identity. Let $f$ be a class function in $A_R(G)$ and fix $r \in R$. 
The set $f^{-1}(r)\cap G^{(0)}$ is a 
compact, open, invariant subset of $G^{(0)}$.
\end{lemma}
\begin{proof}
Without loss of generality, assume that there exists $u \in G^{(0)}$ such that $f(u) =r$; otherwise, $f^{-1}(r)\cap G^{(0)}$ 
would be empty. Write 
$f = \sum\limits_{B \in F} a_B1_B$ where $F$ is a finite collection of mutually disjoint compact open bisections
using Lemma~\ref{lem:3.5}. 
Then we have \[
              f^{-1}(r) = \bigcup\limits_{\{D \in F:a_D = r\}}D.
             \]
Thus $f^{-1}(r)$ is compact open and since $\go$ is clopen we have that $f^{-1}(r) \cap \go$ is compact open.

To see the invariance we show that 
\[
 r(s^{-1}(f^{-1}(r) \cap \go )) \subseteq f^{-1}(r) \cap \go.
\]
Let $u$ be an element of the right-hand side.  Thus, we can find $\gamma \in G$ and such that 
$s(\gamma) \in f^{-1}(r)\cap G^{(0)}$ and $r(\gamma)=u$. 
Thus $f(s(\gamma)) = r$.   Since $f$ is a class function and $s(s(\gamma)) = r(s(\gamma)) = s(\gamma)$ it follows that
 \[f(\gamma\gamma^{-1}) = f(\gamma s(\gamma)\gamma^{-1}) = f(s(\gamma)) = r.\] 
Hence $u=r(\gamma) = \gamma\gamma^{-1} \in f^{-1}(r)\cap G^{(0)}$.
\end{proof}

The following theorem generalises the characterisation of the center of a Kumjian-Pask algebra given in \cite[Theorem~4.7]{center}. 

\begin{theorem}\label{thm:center}
Let $G$ be  a Hausdorff, ample groupoid, and $R$ a commutative ring with identity. 
\begin{enumerate}
\item\label{it1:center} Suppose $G$ is effective and minimal, and $G^{(0)}$ is compact. Then \[Z(A_R(G)) = R1_{G^{(0)}}.\]
\item\label{it2:center} Suppose $G$ is minimal and $G^{(0)}$ is not compact. Then \[Z(A_R(G)) = \{0\}.\]
\end{enumerate}
\end{theorem}

\begin{proof}

(\ref{it1:center}) Since $A_R(G)$ is unital (see \cite[Proposition 4.11]{Steinberg2010}),   $Z(A_R(G)) \neq \{0\}$. 
Let $f \in Z(A_R(G)) \setminus \{0\}$, 
then by \cite[Proposition 4.13]{Steinberg2010} $f$ is a nonzero class function. We claim that
$\supp(f) \setminus G^{(0)} = \emptyset.$  To see this, assume by way of contradiction
 that $\supp(f) \setminus G^{(0)}$ is nonempty. As $\supp(f)$ is open and $G^{(0)}$ is closed we have that 
$\supp(f) \setminus G^{(0)}$ is open and therefore contains a compact open bisection $B$ as $G$ is ample. 
Since $G$ is effective, there exists $\gamma \in B $ 
such that $r(\gamma) \neq s(\gamma)$.  But $f(\gamma) \neq 0$ contradicts item (\ref{it1:classfunction}) of the definition
 of a class function.  Thus $\supp(f) \subseteq G^{(0)}$.  

As $f$ is nonzero there exists $r \in R$ such that $f^{-1}(r)\cap G^{(0)} \neq \emptyset$.  Applying Lemma~\ref{thm:inv}
we see that $f^{-1}(r)\cap G^{(0)}$ is a nonempty open invariant subset of $G^{(0)}$ and hence must 
equal to $G^{(0)}$ as $G$ is minimal. Therefore $f = r1_{G^{(0)}}$.

(\ref{it2:center}) By way of contradiction, suppose there exists
$f \in Z(A_R(G)) \setminus \{0\}$.  Then by \cite[Proposition 4.13]{Steinberg2010} $f$ is a nonzero class function.
Fix $\gamma \in \supp f$ and let $u:=s(\gamma)$.  Since $G$ is minimal, $[u]$ is dense in $\go$.  
Since $\go$ is not compact, $[u]$ is not relatively compact (that is, does not have compact closure).  
For each $u_i \in [u]$, there exists
$\gamma_i \in G$ such that $s(\gamma_i)=u$ and $r(\gamma_i) = u_i$.  Note that because $f$ is a class 
function, condition (\ref{it2:classfunction}) of the definition tells us that $\gamma_i\gamma \gamma_i^{-1} \in \supp f$.
We claim that  $A:=\{\gamma_i\gamma \gamma_i^{-1}\}$ is not relatively compact.
To prove the claim, suppose otherwise.  That is suppose $\overline{A}$ is compact.
Note that $[u] \subseteq r(\overline{A})$.  But since $r$ is continuous, $r(\overline{A})$ is compact and
hence $[u]$ is relatively compact which is a contradiction, proving the claim.
However, we now have a contradiction because $\supp f$ is compact.
Therefore $Z(A_R(G)) = \{0\}$.
\end{proof}

\end{document}